\documentclass[review]{elsarticle}

\usepackage{amsmath}
\usepackage{amssymb}
\usepackage{mathrsfs}
\usepackage{graphicx}
\usepackage{amsmath}
\usepackage{latexsym}
\usepackage{amsfonts}
\usepackage{multirow}
\usepackage{cases}
\usepackage{tabularx}
\usepackage{booktabs}
\usepackage{color}
\usepackage{setspace}

\newtheorem{theorem}{Theorem}
\newtheorem{definition}{Definition}
\newtheorem{lemma}{Lemma}
\newtheorem{example}{Example}
\newenvironment{proof}{{\textbf{Proof.}}\,}{\hfill \qed\\}
\newtheorem{algorithm}{Algorithm}

\begin{document}

\begin{frontmatter}

\title{Computing the dominant eigenpair of an essentially nonnegative tensor via a homotopy method\tnoteref{mytitlenote}}
\tnotetext[mytitlenote]{This
work was supported by the National Natural Science Foundation of
China (Grant No. 11771244).}


\author{Xingbang Cui \quad\quad Liping Zhang\corref{mycorrespondingauthor}}
\address{Department of Mathematical Sciences, Tsinghua University, Beijing 100084, China}
\cortext[mycorrespondingauthor]{Corresponding author. \emph{Email address}: lipingzhang@tsinghua.edu.cn (Liping Zhang)}

\begin{abstract}
The theory of eigenvalues and eigenvectors is one of the fundamental
and essential components in tensor analysis. Computing the dominant eigenpair of an essentially nonnegative tensor is an important topic in tensor
computation because of the critical applications in network resource allocations. In this paper, we consider the aforementioned topic and there are two main contributions. First, we show that an irreducible essentially nonnegative tensor has a unique positive dominant eigenvalue with a unique positive normalized eigenvector.  Second, we present a homotopy method to compute the dominant eigenpair and prove that it converges to
the desired dominant eigenpair whether the given tensor is irreducible or reducible based on an approximation technique. Finally, we implement the method using
a prediction-correction approach for path following and some numerical results are reported
to illustrate the efficiency of the proposed algorithm.
\end{abstract}


\end{frontmatter}


\section{Introduction}

The purpose of this paper is to study the eigenvalue problem of a special class of tensors. An $m$th order $n$-dimensional real tensor is a hypermatrix of $n^m$ elements, which takes the form
$$
\mathcal{A}=(a_{i_1i_2\ldots i_m}),\quad a_{i_1i_2\ldots i_m}\in \mathbb{R}, \quad i_j\in [n],\quad j\in[m],
$$
where $[n]=\{1,2,\ldots,n\}$. When $m=2$, $\mathcal{A}$ is a matrix in $\mathbb{R}^{n\times n}$. The set of such all $m$th order $n$-dimensional real tensors is denoted as $\mathbb{R}^{[m,n]}$.  A tensor $\mathcal{A}\in\mathbb{R}^{[m,n]}$ is called {\it nonnegative} or {\it positive} if its entries $a_{i_1\ldots i_m}\ge 0$ or $a_{i_1\ldots i_m}> 0$ for all $i_j\in [n]$ and $j\in[m]$. We use $\mathbb{R}_+^{[m,n]}$ to denote the set of all nonnegative tensors. A tensor $\mathcal{A}\in\mathbb{R}^{[m,n]}$ is called {\it symmetric} if its entries $a_{i_1\ldots i_m}$
are invariant under any permutation of their indices $\{i_1\ldots
i_m\}$ \cite{qi05}. The set of all $m$th order $n$-dimensional real symmetric tensors is denoted as $\mathbb{S}^{[m,n]}$. Various properties and applications of tensors, nonnegative tensors in particular, can be found in \cite{Kolda,QiLuo}.

The computation of tensor eigenvalues is of crucial importance in a variety of practical
problems in physics and engineering such as blind source separation \cite{Kofidis} and magnetic resonance
imaging \cite{QiChenChen,QiLuo}. Tensor eigenvalues and eigenvectors have received much attention
lately. Especially, finding the maximum eigenvalue of a tensor is an important topic in tensor
computation and multilinear algebra \cite{QiChenChen,QiLuo,bookwei}. The concept of tensor eigenvalues was independently proposed in \cite{lim05,qi05}. More details on eigenvalues of tensors can be found in \cite{Friedland,QiLuo,yy}.
In this paper, we specifically use the following definition of a tensor eigenpair.

\begin{definition} \label{defeig}
Let $\mathcal{A}\in \mathbb{R}^{[m,n]}$ and  $x\in \mathbb{C}^n$. We define
two $n$-dimensional column vectors
$$\mathcal{A}x^{m-1}:=\left(\sum_{{i_2},\ldots,{i_m}=1}^n a_{i{i_2}\ldots{i_m}}x_{i_2}\cdots{x_{i_m}}\right)_{1\le i\le n}\quad \mbox{and}\quad x^{[m-1]}:=\left(x_i^{m-1}\right)_{1\le i\le n}.$$
If the polynomial system
\begin{equation}\label{eqeig}
\mathcal{A}x^{m-1}=\lambda x^{[m-1]}
\end{equation}
has a solution
$(\lambda,x)\in \mathbb{C}\times (\mathbb{C}^n\backslash \{0\})$, then $\lambda$ is called an
eigenvalue of $\mathcal{A}$ and $x$ its corresponding eigenvector.

We call
$\rho(\mathcal{A})$ the spectral radius of tensor $\mathcal{A}$ if
\begin{displaymath}
\rho(\mathcal{A})=\max\{|\lambda|:\, \text{$\lambda$ is an
eigenvalue of $\mathcal{A}$}\},
\end{displaymath}
where $|\lambda|$ denotes the modulus of $\lambda$.
\end{definition}

Nonnegative tensors, arising from multilinear pagerank, spectral hypergraph theory,
and higher-order Markov chains, form a singularly important class of tensors and
have attracted more and more attention because they share some intrinsic properties with those of
the nonnegative matrices. One of those properties is the Perron-Frobenius theorem on eigenvalues.
The generalization of the Perron-Frobenius theorem
to nonnegative tensors can be found in \cite{ChangPZhang,Friedland,yy}. We recall the concept of irreducibility of a tensor and the Perron-Frobenius theorem for nonnegative tensors \cite{QiLuo}. In the sequel,  let $\mathbb{R}_{++}$ denote the set of all positive real numbers and $\mathbb{R}^n_{++}$ denote the set of all positive column vectors.

\begin{definition} 
A tensor $\mathcal{A}\in \mathbb{R}^{[m,n]}$ is said to be {\it reducible}, if there is a nonempty proper index subset $J\subset[n]$ such that
\begin{equation*}
\mathcal{A}_{i_{1}i_{2}\cdots i_{m}}=0,\ \forall i_{1}\in J,\ \forall i_{2},\cdots,i_{m}\in [n]\backslash J.
\end{equation*}
We say that $\mathcal{A}$ is {\it irreducible} if it is not reducible.
\end{definition}

\begin{theorem} \label{thmpf}Let $\mathcal{A}\in\mathbb{R}^{[m,n]}$ and $\lambda_*=\rho(\mathcal{A})$. If $\mathcal{A}\in\mathbb{R}^{[m,n]}_+$, then $\lambda_*$ is an eigenvalue of $\mathcal{A}$ and it has a nonnegative corresponding eigenvector $x_*$.
If furthermore $\mathcal{A}$ is irreducible, then the following hold:
\begin{itemize}
\item[{\rm(i)}] $\lambda_*>0$ and $x_*\in \mathbb{R}^n_{++}$.
\item[{\rm(ii)}] If $\lambda$ is an eigenvalue with a nonnegative eigenvector, then $\lambda=\lambda_*$. Moreover, the nonnegative eigenvector is unique up to a constant multiple.
\end{itemize}
\end{theorem}

Under normalization $x_*^Tx_*=1$, the eigenvector $x_*$ is unique. By Theorem \ref{thmpf}, if  a tensor $\mathcal{A}\in\mathbb{R}^{[m,n]}_+$ is irreducible, then it has a unique positive eigenpair $(\lambda_*,x_*)$, which is called {\it Perron pair}. The Perron pair of a nonnegative tensor $\mathcal{A}$  plays an important role in various
applications, such as in the spectral hypergraph theory, higher order Markov chains, and automatic control, to name a few  \cite{QiLuo}. Several algorithms have been proposed for finding the Perron pair of a nonnegative tensor in the literature. Ng, Qi, and
Zhou \cite{NgQiZhou} proposed a power-type method. The convergence of this method is established in \cite{Chang2011,zlplinear}. Modified versions of this method have been proposed in \cite{Liu2010,Zhang2012,Zhou2013}. To achieve faster convergence, Ni and Qi \cite{NiQi} proposed Newton-type algorithm to solve a polynomial system. Their
algorithm is proved to be locally quadratically convergent when the nonnegative tensor is irreducible.  Liu, Guo, and Lin \cite{LiuGuoLin} recently proposed an algorithm
that combines Newton's and Noda's iterations for third order nonnegative tensors. This algorithm preserves positivity and
is shown to be quadratically convergent to the Perron pair for irreducible nonnegative tensors. Based on Definition \ref{defeig}, finding the eigenpair is equivalent to solving polynomial system (\ref{eqeig}). Since the homopoty method is an attractive class of  methods for solving polynomial systems \cite{Chen2016}, Chen {\it et al} \cite{Chen2019} proposed a homotopy method, based on the algorithm in \cite{Chen2016}, to compute the largest eigenvalue and a
corresponding eigenvector of a nonnegative tensor. They proved that it converges to the
desired eigenpair when the tensor is irreducible.

In this paper, we consider a new class of tensors called essentially nonnegative tensors \cite{Zhang2013},
which extends the nonnegative tensors. We focus on computing the dominant eigenvalue of an essentially nonnegative tensor, which is closely related to the largest eigenvalue of a nonnegative tensor.  The dominant eigenvalue of an essentially nonnegative tensor has many important applications in network resource allocation \cite{ToRoy18} and trace-preserving problems \cite{Zhang2013}. Zhang {\it et al} \cite{Zhang2013} first introduced this class of tensors and studied its maximum real eigenvalue, which is called dominant eigenvalue. They also proposed a power-type method to compute the dominant eigenvalue of an essentially nonnegative tensor in \cite{Zhang2013}, which has linear convergence at most. Hu {\it et al} \cite{Hu2013} proposed a semidefinite programming method to find the dominant eigenvalue of an essentially nonnegative symmetric tensor via solving a sum of squares
of polynomials optimization problem. The method in \cite{Hu2013} are not suitable
for finding the dominant eigenvalue of a large size essentially nonnegative tensor. In this paper, based on the homotopy method in \cite{Chen2019}, we propose a homotopy method for computing the dominant eigenvalue of an essentially
nonnegative tensor. This method is suitable for large size tensors.

This paper is organized as follows. In Section 2, we show that an irreducible essentially nonnegative tensor has the unique dominant eigenvalue with the corresponding unique positive eigenvector. In Section 3, we propose a homotopy method and prove its convergence. We
describe an implementation of the method and give some numerical results in Section 4. Some conclusions are given
in Section 5.

\section{Preliminaries and the unique dominant eigenpair}
\label{sec:1}
We start this section with some fundamental notions and properties
on tensors, see \cite{QiLuo} for more details. We also introduce an important result for eigenvalues of an essentially nonnegative tensor.

The $m$th order $n$-dimensional {\it unit tensor}, denoted by $\mathcal{I}$, is the tensor whose entries are
$\delta_{i_1\ldots i_m}$ with $\delta_{i_1\ldots i_m}=1$ if and only if $i_1=\cdots=i_m$ and otherwise zero. The symbol $\mathcal{A}\ge \mathcal{B}$ means that $\mathcal{A}-\mathcal{B}$ is a nonnegative tensor. We recall the concept of an essentially nonnegative tensor which was first introduced in \cite{Zhang2013}.
\begin{definition} Let $\mathcal{A}\in\mathbb{R}^{[m,n]}$. We say that $\mathcal{A}$ is an {\it essentially nonnegative tensor} if and only if its off-diagonal entries are all nonnegative.
\end{definition}

From the definition, clearly any nonnegative tensor is essentially nonnegative, while the converse
may not be true in general. When the order $m=2$, the definition collapses to the
classical definition of essentially nonnegative matrices \cite{Horn}.
Similar to the essentially nonnegative matrix, an essentially nonnegative tensor $\mathcal{A}$ has a real eigenvalue
with the property that it is greater than or equal to the real part of every eigenvalue of $\mathcal{A}$. The following result was given in \cite{Zhang2013}.
\begin{theorem} \label{thmdoeig} Let $\mathcal{A}\in\mathbb{R}^{[m,n]}$ be an essentially nonnegative tensor and
\begin{equation}\label{alph}\alpha=\max_{i\in[n]}|a_{i\ldots i}|+1.
\end{equation} Then we have
$\mathcal{A}+\alpha\mathcal{I}\in \mathbb{R}^{[m,n]}_+$. Moreover, $\mathcal{A}$ has a real eigenvalue $\lambda(\mathcal{A})$ with corresponding nonnegative eigenvector and $\lambda(\mathcal{A})\ge {\rm Re}\lambda$ for every eigenvalue $\lambda$ of $\mathcal{A}$. Furthermore,
\begin{equation*}\label{relation}
\lambda(\mathcal{A})=\rho(\mathcal{A}+\alpha\mathcal{I})-\alpha.
\end{equation*}
\end{theorem}

We call such an eigenvalue $\lambda(\mathcal{A})$ in Theorem \ref{thmdoeig} the {\it dominant eigenvalue} of an essentially nonnegative tensor $\mathcal{A}$.

We also use the following results in the sequel, which can be found in \cite{yy,Zhang2013,zhangmtensor}.
\begin{lemma}   \label{lem2.2}
Let  $\mathcal{A}\in \mathbb{R}^{[m,n]}_+$  and $\varepsilon>0$ be a sufficiently small number. If
$\mathcal{A}_\varepsilon= \mathcal{A}+ \mathcal{E}$ where
$\mathcal{E}$ denotes the tensor with every entry being
$\varepsilon$, then
\begin{equation*}
\lim_{\varepsilon\downarrow
0}\rho(\mathcal{A}_\varepsilon)=\rho(\mathcal{A}).
\end{equation*}
If furthermore $\mathcal{A}\in \mathbb{S}^{[m,n]}$, then
$$
0\le \rho(\mathcal{A}_\varepsilon)-\rho(\mathcal{A})\le \varepsilon n^{m-1}.
$$
\end{lemma}

\begin{lemma}\label{lem2.3}
Let $\mathcal{A},\mathcal{B}\in \mathbb{R}^{[m,n]}_+$. If $\mathcal{B}$ is irreducible, $\mathcal{A}\le \mathcal{B}$ and $\mathcal{A}\neq \mathcal{B}$, then $\rho(\mathcal{A})<\rho(\mathcal{B})$.
\end{lemma}

By Lemma \ref{lem2.3} and Theorems \ref{thmpf} and \ref{thmdoeig}, we have the following theorem.
\begin{theorem}\label{uniqthm}
Let $\mathcal{A}\in\mathbb{R}^{[m,n]}$ be an essentially nonnegative tensor and $\alpha$ be defined by (\ref{alph}). Then $\mathcal{A}$ has a nonnegative eigenpair $(\lambda(\mathcal{A}),x)$. If furthermore $\mathcal{A}$ is irreducible, then every nonnegative eigenpair $(\lambda(\mathcal{A}),x)$ is positive. Moreover, $\mathcal{A}$ has a unique positive dominant eigenvalue $\lambda(\mathcal{A})$, and the nonnegative eigenvector is unique  up to a constant multiple.
\end{theorem}
\begin{proof} By Theorems \ref{thmpf} and \ref{thmdoeig}, $\lambda(\mathcal{A})\ge 0$ and it corresponds to a nonnegative eigenvector. Let $\mathcal{B}=\alpha\mathcal{I}$ and $\mathcal{C}=\mathcal{A}+\alpha\mathcal{I}$. Then $\mathcal{B}\le \mathcal{C}$ and $\mathcal{B}\ne \mathcal{C}$.
Since $\mathcal{A}$ is irreducible, $\mathcal{C}$ is also irreducible. By Lemma \ref{lem2.3}, we have
$$
\alpha=\rho(\mathcal{B})<\rho(\mathcal{C}),
$$
which implies that $\lambda(\mathcal{A})>0$. By Theorem \ref{thmpf}, $\mathcal{A}$ has a unique positive normalized eigenpair.
\end{proof}

Theorem \ref{uniqthm} shows that an irreducible essentially nonnegative tensor has a unqiue positive eigenpair. We call it {\it dominant eigenpair}, denoted as $(\lambda(\mathcal{A}),x(\mathcal{A}))$. According to Theorem \ref{thmdoeig} and \ref{uniqthm}, eigenpair $(\lambda(\mathcal{A})+\alpha,x(\mathcal{A}))$ coincides with Perron pair $(\rho(\mathcal{A}+\alpha\mathcal{I}),x)$ for nonnegative tensor $\mathcal{A}+\alpha\mathcal{I}$ with eigenvectors normalized, which can be calculated via the homotopy method in \cite{Chen2019}. This is the core of motivation of our paper. In the next section, we will propose a homotopy method to compute the  dominant eigenpair for an essentially nonnegative tensor.

\section{A homotopy method}

In this section, our goal is to compute the dominant eigenpair of an irreducible essentially nonnegative tensor $\mathcal{A}$ by a homotopy method. We first assume that $\mathcal{A}$ is irreducible and let $\mathcal{T}=\mathcal{A}+\alpha\mathcal{I}$ where $\alpha$ is defined by (\ref{alph}).

For the purpose of computing the dominant eigenpair of $\mathcal{A}$, we will solve the following problem
\begin{equation}\label{Qlx}
Q(\lambda,x)=\begin{pmatrix}
\mathcal{T}x^{m-1}-\lambda x^{[m-1]}\\
x^{T}x-1
\end{pmatrix}=0.
\end{equation}
Choose $a,b\in \mathbb{R}^n_{++}$ and define a positive tensor
\begin{equation}\label{that}
\mathcal{S}=a^{[m-1]}\circ b\circ\cdots\circ b\in \mathbb{R}^{[m,n]}
\end{equation}
with its entry $s_{i_1i_2\ldots i_m}=a_{i_1}^{m-1}b_{i_2}\cdots b_{i_m}$, where $\circ$ denotes the outer product. Clearly,
\begin{equation}\label{eqzero}
\lambda_0=(a^Tb)^{m-1}, \quad x_0=\frac{a}{\|a\|_2}
 \end{equation}
 is the Perron pair of $\mathcal{S}$ \cite[Lemma 2.1]{Chen2019}. For the positive tensor $\mathcal{S}$, we give the start system
\begin{equation}\label{Plx}
P(\lambda,x)=\begin{pmatrix}
\mathcal{S}x^{m-1}-\lambda x^{[m-1]}\\
x^{T}x-1
\end{pmatrix}=0.
\end{equation}
Thus, we construct the following homotopy
\begin{equation}\label{Htau}
H_{\tau}(\lambda,x)=(1-\tau)P(\lambda,x)+\tau Q(\lambda,x)=0,\quad \tau\in[0,1],
\end{equation}
i.e.,
$$
H_{\tau}(\lambda,x)=\begin{pmatrix}
(\tau\mathcal{T}+(1-\tau)\mathcal{S})x^{m-1}-\lambda x^{[m-1]}\\
x^{T}x-1
\end{pmatrix}=0.
$$

In designing a homotopy method for computing the dominant eigenpair, the Jacobian matrix of $H_{\tau}(\lambda,x)$ plays an important role. We now focus on computing this matrix.
By \cite[Lemma 2.1]{NiQi}, using the semi-symmetric technique given in \cite{NiQi}, there exists the unique semi-symmetric tensor ${\mathcal{A}}_s\in \mathbb{R}^{[m,n]}$ for a general tensor $\mathcal{A}\in \mathbb{R}^{[m,n]}$ such that ${\mathcal{A}}_sx^{m-1}=\mathcal{A}x^{m-1}$
for all $x\in \mathbb{R}^n$.  By \cite[Lemma 3.3]{NiQi}, the Jacobian matrix of $\mathcal{A}x^{m-1}$  is given by
\begin{equation}\label{JF}
J\mathcal{A}x^{m-1}=(m-1){\mathcal{A}}_sx^{m-2}.
\end{equation}
Here, for a tensor $\mathcal{A}=(a_{i_1\ldots i_m})\in \mathbb{R}^{[m,n]}$ and a
vector $x\in \mathbb{R}^n$,  $\mathcal{A}{x}^{m-2}$ is a matrix
in $\mathbb{R}^{n\times n}$ whose $(i,j)$-th component is defined by
$$
\left(\mathcal{A}{x}^{m-2}\right)_{ij}=\sum_{i_3=1}^n\cdots
\sum_{i_m=1}^n a_{iji_3\ldots i_m}x_{i_3}\cdots x_{i_m}.
$$
Note that the tensor $\mathcal{S}$ is semi-symmetric. Thus, it follows from (\ref{JF}) that the Jacobian matrix $JH_{\tau}(\lambda,x)\in \mathbb{R}^{(n+1)\times(n+1)}$ of  $H_{\tau}(\lambda,x)$ is given by
\begin{equation}\label{JH}
JH_{\tau}(\lambda,x)=\begin{pmatrix}
-x^{[m-1]}\quad\quad\quad & (m-1)[\tau{\mathcal{T}}_s+(1-\tau)\mathcal{S}-\lambda\mathcal{I}]x^{m-2}\\
0 \quad\quad\quad & 2x^T
\end{pmatrix},
\end{equation}
where ${\mathcal{T}}_s$ is the corresponding semi-symmetric tensor of ${\mathcal{T}}$. Hence, we have the following theorem.
\begin{theorem} \label{thmrank}Let $\mathcal{A}\in\mathbb{R}^{[m,n]}$ be an essentially nonnegative tensor and $\mathcal{T}=\mathcal{A}+\alpha\mathcal{I}$ where $\alpha$ is defined by (\ref{alph}). Let $a,b\in \mathbb{R}_{++}$ and the positive tensor $\mathcal{S}$ be defined by (\ref{that}). Then, we have the following results.
\begin{itemize}
\item[{\rm(i)}] For any $\tau\in[0,1)$, $H_{\tau}(\lambda,x)=0$, defined by (\ref{Htau}),  has a unique solution $(\lambda(\tau),x(\tau))\in \mathbb{R}_{++}\times \mathbb{R}_{++}^n$, which is the Perron pair of the positive tensor $\tau\mathcal{T}+(1-\tau)\mathcal{S}$. Moreover, the Jacobian matrix $JH_{\tau}(\lambda(\tau),x(\tau))$ defined as (\ref{JH}), i.e.,
 \begin{equation*}
JH_{\tau}(\lambda(\tau),x(\tau))=\begin{pmatrix}
-x(\tau)^{[m-1]}\quad\quad\quad & (m-1)[\tau{\mathcal{T}}_s+(1-\tau)\mathcal{S}-\lambda(\tau)\mathcal{I}]x(\tau)^{m-2}\\
0 \quad\quad\quad & 2x(\tau)^T
\end{pmatrix},
\end{equation*}
     is nonsingular.
\item[{\rm(ii)}] If furthermore $\mathcal{A}$ is irreducible, then $H_1(\lambda,x)=Q(\lambda,x)=0$ has a unique solution $(\lambda_*,x_*)=(\lambda(1),x(1))\in \mathbb{R}_{++}\times \mathbb{R}_{++}^n$. Moreover, the Jacobian matrix
$$JH_1(\lambda_*,x_*)=\begin{pmatrix}
-x_*^{[m-1]}\quad\quad\quad & (m-1)[{\mathcal{T}}_s-\lambda_*\mathcal{I}]x_*^{m-2}\\
0 \quad\quad\quad & 2x_*^T
\end{pmatrix}
$$ is nonsingular, and $(\lambda_*-\alpha,x_*)$ is the unique positive dominant eigenpair of $\mathcal{A}$.
\item[{\rm(iii)}] Let $\mathcal{T}=\mathcal{A}_\varepsilon+\alpha\mathcal{I}$ when $\mathcal{A}$ is reducible, where $\mathcal{A}_\varepsilon$ is defined in Lemma \ref{lem2.2}. Then $\mathcal{T}$ is irreducible, $H_1(\lambda,x)=Q(\lambda,x)=0$ has a unique solution $(\lambda^{\varepsilon}_*,x^{\varepsilon}_*)=(\lambda^{\varepsilon}(1),x^{\varepsilon}(1))\in \mathbb{R}_{++}\times \mathbb{R}_{++}^n$. Moreover, the Jacobian matrix
$$JH^{\varepsilon}_1(\lambda^{\varepsilon}_*,x^{\varepsilon}_*)=\begin{pmatrix}
-(x^{\varepsilon}_*)^{[m-1]}\quad\quad\quad & (m-1)[{\mathcal{T}}_s-\lambda^{\varepsilon}_*\mathcal{I}](x^{\varepsilon}_*)^{m-2}\\
0 \quad\quad\quad & 2(x^{\varepsilon}_*)^T
\end{pmatrix}
$$ is nonsingular, and $(\lambda^{\varepsilon}_*-\alpha,x^{\varepsilon}_*)$ is the unique positive dominant eigenpair of $\mathcal{A}_{\varepsilon}$. Especially,
$$\lim_{\varepsilon\downarrow 0}\lambda^{\varepsilon}_*-\alpha=\lambda(\mathcal{A}),\quad \lim_{\varepsilon\downarrow 0}x^{\varepsilon}_*=x(\mathcal{A}),
$$
where $(\lambda(\mathcal{A}),x(\mathcal{A}))$ is a dominant eigenpair of $\mathcal{A}$.
\end{itemize}
\end{theorem}
\begin{proof} For part (i), since $\tau\mathcal{T}+(1-\tau)\mathcal{S}$ is positive, it is irreducible. By Theorem \ref{thmpf},
$$
(\tau\mathcal{T}+(1-\tau)\mathcal{S})x^{m-1}-\lambda x^{[m-1]}=0
$$
has a unique solution in $\mathbb{R}_{++}\times \mathbb{R}_{++}^n$, up to a
constant multiple of $x$. By imposing the normalization condition $x^Tx=1$, $H_{\tau}(\lambda,x)=0$ has a unique solution $(\lambda(\tau),x(\tau))$ in $\mathbb{R}_{++}\times \mathbb{R}_{++}^n$. Clearly,  $(\lambda(\tau),x(\tau))$ is the Perron pair of tensor $\tau\mathcal{T}+(1-\tau)\mathcal{S}$.

In order to show the nonsingularity of $JH_{\tau}(\lambda(\tau),x(\tau))$, we set
$$
\begin{pmatrix}
-x(\tau)^{[m-1]}\quad\quad\quad & (m-1)[\tau{\mathcal{T}}_s+(1-\tau)\mathcal{S}-\lambda(\tau)\mathcal{I}]x(\tau)^{m-2}\\
0 \quad\quad\quad & 2x(\tau)^T
\end{pmatrix}\begin{pmatrix}t\\
z
\end{pmatrix}=0.
$$
Then we have
\begin{equation}\label{eqmath1}
-tx(\tau)^{[m-1]}+\left((m-1)[\tau{\mathcal{T}}_s+(1-\tau)\mathcal{S}-\lambda(\tau)\mathcal{I}]x(\tau)^{m-2}\right)z=0
\end{equation}
and
\begin{equation}\label{eqmath2}
2x(\tau)^Tz=0.
\end{equation}
Multiplying (\ref{eqmath1}) by $x(\tau)$ in left-side, we have
$$
\left((m-1)[\tau{\mathcal{T}}_s+(1-\tau)\mathcal{S}-\lambda(\tau)\mathcal{I}]x(\tau)^{m-1}\right)^Tz-tx(\tau)^T(x(\tau)^{[m-1]})=0,
$$
which, together with $[\tau{\mathcal{T}}_s+(1-\tau)\mathcal{S}-\lambda(\tau)\mathcal{I}]x(\tau)^{m-1}=0$ and $x(\tau)\in \mathbb{R}^n_{++}$, implies that $t=0$. Hence, (\ref{eqmath1}) reduces to
$$
\left((m-1)[\tau{\mathcal{T}}_s+(1-\tau)\mathcal{S}-\lambda(\tau)\mathcal{I}]x(\tau)^{m-2}\right)z=0.
$$
Since the semi-symmetric tensor ${\mathcal{T}}_s$ of $\mathcal{T}$ is also nonnegative and $\mathcal{S}$ is positive, by the proof of \cite[Lemma 4.1]{NiQi}, the matrix $\left((m-1)[\tau{\mathcal{T}}_s+(1-\tau)\mathcal{S}-\lambda(\tau)\mathcal{I}]x(\tau)^{m-2}\right)$ is irreducible. Therefore, $z=0$ follows from (\ref{eqmath2}). Hence, the Jacobian matrix $JH_{\tau}(\lambda(\tau),x(\tau))$ is nonsingular.

The proof of part (ii) is similar. By Theorems \ref{thmdoeig} and \ref{uniqthm}, we show that $(\lambda_*-\alpha,x_*)$ is the unique positive dominant eigenpair of $\mathcal{A}$.

For part (iii), by the definition of $\mathcal{A}_{\varepsilon}$, $\mathcal{A}_{\varepsilon}+\alpha\mathcal{I}$ is positive and so it is irreducible. Similar to the proof of part (ii), we prove part (iii) from Lemma \ref{lem2.2}.
\end{proof}

Analogous to Lemmas 2.3 and 2.4 in \cite{Chen2019}, we have the following lemma.
\begin{lemma}\label{lem2.4}
Let $\mathcal{A}\in\mathbb{R}^{[m,n]}$ be an essentially nonnegative tensor and $\mathcal{T}=\mathcal{A}+\alpha\mathcal{I}$ where $\alpha$ is defined by (\ref{alph}). Let $a,b\in \mathbb{R}^n_{++}$ and the positive tensor $\mathcal{S}$ be defined by (\ref{that}). Then the set
$$\Omega=\{(\lambda,x,\tau)\in \mathbb{R}_{++}\times \mathbb{R}_{++}^n\times[0,1)\,|~H_{\tau}(\lambda,x)=0\}
$$
is a
one-dimensional smooth manifold and is uniformly bounded for $\tau\in[0,1)$.
\end{lemma}

By (i) and (ii) of Theorem \ref{thmrank}, it is reasonable that we can find the unique positive dominant eigenpair of an irreducible essentially nonnegative tensor by following the curve $(\lambda(\tau),x(\tau))$ with $\tau\uparrow 1$. When the given tensor is reducible, by Theorem \ref{thmrank}, we also can find a $\varepsilon$-approximal dominant eigenpair with $\tau\uparrow 1$ and then a dominant eigenpair can be found from Lemma \ref{lem2.2}. The following theorem shows that the homotopy (\ref{Htau}) works well.
\begin{theorem}\label{thmmain}
Let $\mathcal{A}\in\mathbb{R}^{[m,n]}$ be an essentially nonnegative tensor and $\mathcal{T}=\mathcal{A}+\alpha\mathcal{I}$ where $\alpha$ is defined by (\ref{alph}). Let $a,b\in \mathbb{R}^n_{++}$ and the positive tensor $\mathcal{S}$ be defined by (\ref{that}). Starting from the Perron pair $(\lambda_0,x_0)$ of $\mathcal{S}$ as defined in (\ref{eqzero}), solve the homotopy $H_{\tau}(\lambda,x)=0$ in $\mathbb{R}_{++}\times \mathbb{R}_{++}^n\times[0,1)$, and let $(\lambda(\tau),x(\tau))$ be the generated solution curve. Then every limit point of $(\lambda(\tau)-\alpha,x(\tau))$ is a nonnegative dominant eigenpair of $\mathcal{A}$.

If $\mathcal{A}$ is irreducible, then
$$
\lim_{\tau\uparrow1}\lambda(\tau)-\alpha=\lambda(\mathcal{A}),\quad \lim_{\tau\uparrow1}x(\tau)=x(\mathcal{A}),
$$
where $(\lambda(\mathcal{A}),x(\mathcal{A}))$ is the unique positive dominant eigenpair of $\mathcal{A}$.

If $\mathcal{A}$ is reducible, let $\mathcal{T}=\mathcal{A}_\varepsilon+\alpha\mathcal{I}$, where $\mathcal{A}_\varepsilon$ is defined in Lemma \ref{lem2.2}, in the homotopy $H_{\tau}(\lambda,x)=0$. Let $(\lambda^\varepsilon(\tau),x^\varepsilon(\tau))$ be the solution curve obtained by solving this homotopy  in $\mathbb{R}_{++}\times \mathbb{R}_{++}^n\times[0,1)$. Then
$$
\lim_{\varepsilon\downarrow 0}\lim_{\tau\uparrow1}\lambda^\varepsilon(\tau)-\alpha=\lambda(\mathcal{A}),\quad \lim_{\varepsilon\downarrow 0}\lim_{\tau\uparrow1}x^\varepsilon(\tau)=x(\mathcal{A}),
$$
where $(\lambda(\mathcal{A}),x(\mathcal{A}))$ is a nonnegative dominant eigenpair of $\mathcal{A}$.
\end{theorem}
\begin{proof} By Lemma \ref{lem2.4}, the sequence $\{(\lambda(\tau),x(\tau))\}$ is in $\mathbb{R}_{++}\times \mathbb{R}_{++}^n$ and is uniformly bounded for $\tau\in[0,1)$. Hence, $\{(\lambda(\tau),x(\tau))\}$ has at least a limit point as $\tau\uparrow 1$. Let $(\lambda_*,x_*)$ be an any limit point of  $\{(\lambda(\tau),x(\tau))\}$ as $\tau\uparrow 1$. Then $(\lambda_*,x_*)\in \mathbb{R}_+\times \mathbb{R}^n_+$ and satisfies
\begin{equation}\label{eqmath3}
(\mathcal{A}+\alpha\mathcal{I})x_*^{m-1}-\lambda_*x_*^{[m-1]}=0,\quad x_*^Tx_*=1,
\end{equation}
which, together with Theorem \ref{thmdoeig}, implies that $(\lambda_*-\alpha,x_*)$ is a nonnegative dominant eigenpair of $\mathcal{A}$.

If $\mathcal{A}$ is irreducible, by Theorem \ref{uniqthm}, $\mathcal{A}$ has a unique positive dominant eigenpair $(\lambda(\mathcal{A}),x(\mathcal{A}))$. Hence, combining Theorem \ref{thmrank} (ii)  and (\ref{eqmath3}), we have
$$
\lim_{\tau\uparrow1}\lambda(\tau)-\alpha=\lambda(\mathcal{A}),\quad \lim_{\tau\uparrow1}x(\tau)=x(\mathcal{A}).
$$

If $\mathcal{A}$ is reducible, then it has a nonnegative dominant eigenpair from Theorem \ref{thmdoeig}, denoted as $(\lambda(\mathcal{A}),x(\mathcal{A}))$, too.  Since $\mathcal{A}_\varepsilon+\alpha\mathcal{I}$ is positive and so it is irreducible. Hence, by Theorem \ref{thmpf}, $\mathcal{A}_\varepsilon+\alpha\mathcal{I}$ has a unique positive Perron pair, denoted as $(\lambda^{\varepsilon}_*,x^{\varepsilon}_*)$. Therefore, the obtained sequence $\{(\lambda^\varepsilon(\tau),x^\varepsilon(\tau))\}$ converges to $(\lambda^{\varepsilon}_*,x^{\varepsilon}_*)$ as $\tau\uparrow1$, i.e.,
$$
\lim_{\tau\uparrow1}\lambda^\varepsilon(\tau)=\lambda^\varepsilon_*, \quad \lim_{\tau\uparrow1}x^\varepsilon(\tau)=x^\varepsilon_*.
$$
By Theorem \ref{thmrank} (iii), we have
$$\lim_{\varepsilon\downarrow 0}\lambda^{\varepsilon}_*-\alpha=\lambda(\mathcal{A}),\quad \lim_{\varepsilon\downarrow 0}x^{\varepsilon}_*=x(\mathcal{A}).
$$
So we complete the proof.
\end{proof}

Theorems \ref{thmrank} and \ref{thmmain} show that the homotopy method (\ref{Htau}) works well. For computing the dominant eigenpair of an essentially nonnegative tensor, we will propose a detailed algorithm to implement the homotopy method (\ref{Htau}).

\section{Algorithm description and numerical experiments}

We now present an algorithm that implements the homotopy method (\ref{Htau}) for computing
the unique positive dominant eigenpair of an irreducible essentially nonnegative tensor $\mathcal{A}\in\mathbb{R}^{[m,n]}$. In order to follow the curve in the homotopy method, we differentiate $H_\tau(\lambda,x)=0$ with respect to $\tau$. Then,
$$
{\rm D}_\lambda H_\tau\cdot\frac{{\rm d}\lambda}{{\rm d}\tau}+{\rm D}_xH_\tau\cdot\frac{{\rm d}x}{{\rm d}\tau}+{\rm D}_\tau H_\tau=0,
$$
i.e.,
\begin{equation}\label{diffeq}
JH_\tau \cdot \begin{pmatrix}
\frac{{\rm d}\lambda}{{\rm d}\tau}\\
\frac{{\rm d}x}{{\rm d}\tau}
\end{pmatrix}=-{\rm D}_\tau H_\tau.
\end{equation}
By Theorem \ref{thmrank}, $JH_\tau$ is nonsingular for all $\tau\in[0,1)$, the system of differential equations (\ref{diffeq}) is well defined. Similar to the algorithmic framework in \cite{Chen2019}, we follow the curve obtained by solving the system of differential equations (\ref{diffeq}), where an Euler-Newton type predication-correction approach is utilized. Then let $(\lambda(0),x(0))=(\lambda_0,x_0)$ with $(\lambda_0, x_0)$ defined in (\ref{eqzero}). Thus, we can obtain an approximation $(\bar{\lambda},\bar{x})$ to $(\lambda(\tau),x(\tau))$:
$$
\begin{pmatrix}
\bar{\lambda}\\
\bar{x}
\end{pmatrix}=\begin{pmatrix}
{\lambda}\\
{x}
\end{pmatrix}+\Delta \tau \begin{pmatrix}
\frac{{\rm d}\lambda}{{\rm d}\tau}\\
\frac{{\rm d}x}{{\rm d}\tau}
\end{pmatrix},
$$
 where $\Delta \tau$ is a step size such that $\tau$ increases to $1$.

\begin{algorithm}\label{alg31} {\rm

\begin{description}
\item[]
{\bf INPUT.} Given an essentially nonnegative tensor $\mathcal{A}\in\mathbb{R}^{[m,n]}$. Let $\alpha$ be defined by (\ref{alph}) and $\mathcal{T}=\mathcal{A}+\alpha\mathcal{I}$ if $\mathcal{A}$ is irreducible, $\mathcal{T}=\mathcal{A}_\varepsilon+\alpha\mathcal{I}$ otherwise, where $\mathcal{A}_\varepsilon$ is given in Lemma \ref{lem2.2}.

\item[] {\bf Initialization.}  Choose positive vectors $a,b\in \mathbb{R}^n_{++}$ and construct the positive tensor $\mathcal{S}$ as (\ref{that}).
Choose initial step size $\Delta\tau_0>0$, tolerances $\epsilon_1>0$ and $\epsilon_2>0$. Let $\tau_0=0$, $\lambda_0$ and $x_0$ be defined by (\ref{eqzero}). Choose $\beta\in(0,1)$ that is close to $1$. Set $k=0$.

\item[] {\bf Path following.} For $k=0,1,\ldots$ until $\tau_N<\beta$ and $\tau_{N+1}\ge \beta$ for some $N$:

Set $\tau_{k+1}=\tau_k+\Delta\tau_k$. If $\tau_k<\beta$ and $\tau_{k+1}\ge \beta$, then set $N=k$ and reset $\tau_{N+1}=\beta$ and
$\Delta\tau_N=\beta-\tau_N$.

Let $u=(\lambda,x)$ and $u_k=(\lambda_k,x_k)$. We employ the following Euler-Newton
prediction-correction strategy to find the next point on the path $H_\tau(u)=0$:

\begin{itemize}
\item Prediction Step: Compute the tangent vector $g=\frac{{\rm d}u}{{\rm d}\tau}$ to $H_\tau(u)=0$ at $\tau_k$ by solving the linear system
$$
JH_{\tau_k}(u_k)\cdot g =-{\rm D}_\tau H_{\tau_k}(u_k).
$$
Then compute the approximation  $\bar{u}$ to $u_{k+1}$ by
$$
\bar{u}=u_k+\Delta\tau_k\cdot g.
$$

\item Correction Step: Use Newton's iterations. Initialize $v_0=\bar{u}$. For $i=0,1,2,\ldots$, compute
$$
v_{i+1} = v_i-[JH_{\tau_{k+1}}(v_i)]^{-1}H_{\tau_{k+1}}(v_i)
$$
until $\|H_{\tau_{k+1}}(v_J)\|_2 \le \epsilon_1$. Then set $u_{k+1}=v_J$. If $k = N$, go to Endgame.

\item Adaptively updating the step size $\Delta\tau_k$: If more than three steps of Newton iterations were required to converge within
the desired accuracy, then set $\Delta\tau_{k+1}=0.5\Delta\tau_k$. If $\Delta\tau_{k+1}\le 10^{-6}$, set $\Delta\tau_{k+1}=10^{-6}$. If two consecutive steps were not cut, then set $\Delta\tau_{k+1}=2\Delta\tau_k$. If $\Delta\tau_{k+1}\ge 0.4$, set $\Delta\tau_{k+1}=0.4$. Otherwise, $\Delta\tau_{k+1}=\Delta\tau_{k}$.
\end{itemize}
Set $k=k+1$.

\item[]{\bf Endgame.} Set $\Delta\tau=1-\beta$. First compute the tangent vector $g=\frac{{\rm d}u}{{\rm d}\tau}$ to $H_\tau(u)=0$ at $\beta$ by solving the linear system
    $$
JH_{\beta}(u_{N+1})\cdot g =-{\rm D}_\tau H_{\beta}(u_{N+1}),
$$
and compute the prediction $\bar{u}$ by
$$
\bar{u}=u_{N+1}+\Delta\tau\cdot g.
$$
Then use Newton's method to compute the correction. Initialize $v_0=\bar{u}$. For $i=0,1,\ldots$, compute
$$v_{i+1}=v_i-[JH_1(v_i)]^{-1}H_1(v_i).
$$
until $\|Q(v_{J})\|_2\le \epsilon_2$. Stop, and set $(\lambda_*,x_*)=v_J$. Output $(\lambda_*-\alpha,x_*)$ as the dominant eigenpair of $\mathcal{A}$.
\end{description}}
\end{algorithm}

Note that $H_1(v_i)=Q(v_i)$ from (\ref{Qlx}) and (\ref{Htau}) in Endgame of Algorithm \ref{alg31}. By Theroems \ref{thmrank} and \ref{thmmain}, Algorithm \ref{alg31} is well-defined. We now report some numerical results testing Algorithm \ref{alg31}. We compare it with the power-type algorithm proposed in \cite{Zhang2013}, denoted as {\bf PTA} for convenience. All the experiments were done using MATLAB R2015a on a laptop computer with Intel Core i5-6300HQ at 2.3 GHz and 4 GB memory running Microsoft Windows 10. The tensor toolbox of \cite{tool} was used to compute tensor-vector products and to compute
the semi-symmetric tensor $\mathcal{T}_s$.

In our experiments, we used $a=b=(1,\ldots,1)^T$ and $\Delta\tau_0=0.1$, $\epsilon_{1}=10^{-5}$, $\epsilon_{2}=10^{-10}$, and $\beta=0.9999$ in
Algorithm \ref{alg31}. We also used $x_0$ defined in (\ref{eqzero}) as the initial vector in {\bf PTA}. {\bf PTA} was terminated if one of the following conditions was met:
\begin{itemize}
\item[(a)] $\|Q(\lambda_k,x_k)\|_2\le 10^{-10}$, where $Q(\lambda,x)$ is defined as (\ref{Qlx}).
\item[(b)] The number of iterations exceeds $50000$.
\end{itemize}
Note that regular termination condition (a) is the same as the one used in Algorithm \ref{alg31} at regular termination. We also
set the maximal allowed number of prediction-correction steps for Algorithm \ref{alg31} as $50000$. We set $\varepsilon=10^{-9}$ in $\mathcal{A}_\varepsilon$ if required.

We first test Algorithm \ref{alg31} on the
following four examples.
\begin{example} \label{eg1}
This example was given in \cite[Example 5.1]{Zhang2013}. Consider the following essentially nonnegative tensor $\mathcal{A}\in \mathbb{R}^{[3,3]}$ defined by
\begin{align*}
\mathcal{A}(:,:,1)&=\begin{pmatrix}
  -1.51 & 8.35 & 1.03 \\
  4.04 & 3.72 & 1.45 \\
  6.71 & 6.43 & 1.35
\end{pmatrix},
\\
\mathcal{A}(:,:,2)&=\begin{pmatrix}
  9.02 & 0.78 & 6,89 \\
  9.71 & -5.32 & 1.85 \\
  2.09 & 4.17 & 2.98
\end{pmatrix},\\
\mathcal{A}(:,:,3)&=\begin{pmatrix}
  9.55 & 1.57 & 6.91 \\
  5.63 & 5.55 & 1.43 \\
  5.76 & 8.29 & -0.15
\end{pmatrix}.
\end{align*}
\end{example}

\begin{example} \label{eg2} This example was given in \cite[Example 5.2]{Zhang2013}. Consider the following essentially nonnegative tensor $\mathcal{A}\in \mathbb{R}^{[3,3]}$ defined by $\mathcal{A}_{133}=\mathcal{A}_{233}=\mathcal{A}_{311}=\mathcal{A}_{322}=1$, $\mathcal{A}_{111}=\mathcal{A}_{222}=-1$, and zero otherwise.
\end{example}

\begin{example} \label{eg3}
Let $\mathcal{A}\in \mathbb{R}^{[3,100]}$ be a randomly generated tensor. Its diagonal elements are within $[-1,0]$ and  off-diagonal elements are within $[0,1]$.
\end{example}

\begin{example} \label{eg4}
Let $\mathcal{A}\in \mathbb{R}^{[4,50]}$ be a randomly generated tensor. Its diagonal elements are within $[-1,0]$ and  off-diagonal elements are within $[0,1]$.
\end{example}

Clearly, the tensors in Examples \ref{eg1} and  \ref{eg2} are irreducible. The tensors defined in Examples \ref{eg3} and  \ref{eg4} are randomly generated large-scale essentially nonnegative tensors. We summarize the numerical
results for Examples \ref{eg1}-\ref{eg4} in Table \ref{tab1}. In order to show the efficiency of Algorithm \ref{alg31}, we report the number of
prediction-correction steps (\emph{iter}),  the total number of Newton iterations (\emph{nwtiter}), the elapsed CPU time in seconds (\emph{time}), and the dominant eigenvalue $\lambda(\mathcal{A})$ in Table \ref{tab1}.

\begin{table}[h]
\centering
\setlength{\abovecaptionskip}{0pt}%
\setlength{\belowcaptionskip}{10pt}%
\caption{Performance of Algorithm \ref{alg31} on Examples \ref{eg1}-\ref{eg4}}
\begin{tabular}{c c c cc}
\hline
Example & \emph{iter} & \emph{nwtiter} & \emph{time} & $\lambda(\mathcal{A})$\\
\hline
\ref{eg1} & 11 & 18  & 0.6382 & 36.2757\\
\ref{eg2} & 11 & 19  & 0.1599 & 1.0000\\
\ref{eg3} & 11 & 13  & 1.7922 & $5.0021\times10^{3}$\\
\ref{eg4} & 11 & 13  & 6.2493 & $6.2482\times10^{4}$\\
\hline
\end{tabular}
\label{tab1}
\end{table}

From Table \ref{tab1}, we can see that  Algorithm \ref{alg31} performs very well even for the large-scale essentially nonnegative tensors in Examples \ref{eg3} and \ref{eg4}. This shows that the homotopy method (\ref{Htau}) is very effective.

In the following example, we compare Algorithm \ref{alg31} with the power-type algorithm {\bf PTA} proposed in \cite{Zhang2013}.
\begin{example}\label{eg5} Consider an essentially nonnegative tensor $\mathcal{A}\in \mathbb{R}^{[m,n]}$, the diagonal entries of $\mathcal{A}$ are randomly generated from the interval $[-1,0]$ and the off-diagonal elements are in $[0,1]$. For the sake of comparison, all entries of $\mathcal{A}$ are adjusted as follows:
\begin{equation}\label{modif}
a_{i_1\ldots i_m}=a_{i_1\ldots i_m}\times10^{-d},\quad i_j\in [n],~~j\in[m],
\end{equation}
where $d$ is a positive integer. After this modification, we randomly generate $100$ such tensors $\mathcal{A}\in \mathbb{R}^{[m,n]}$ and use this two algorithms to show their performances.
\end{example}

In both algorithms, we set the test tensor $\mathcal{T}=\mathcal{A}_\varepsilon+\alpha\mathcal{I}$, where $\mathcal{A}_\varepsilon$ is defined in Lemma \ref{lem2.2}. Clearly, $\mathcal{T}$ is positive so it is irreducible. Hence, the algorithm {\bf PTA} has linear convergence \cite{Zhang2013,zlplinear,QiLuo} and its rate of convergence from \cite[Theorem 3.88]{QiLuo} is
\begin{equation}\label{rato}
1-\frac{\min_{i,j\in[n]}t_{ij\ldots j}}{\max_{i\in[n]}\sum_{i_2,\ldots,i_m=1}^nt_{ii_2\ldots i_m}}.
\end{equation}
Clearly, after the modification in (\ref{modif}), the off-diagonal elements of $\mathcal{T}$ are very small when $d$ is large. Thus, the rate of convergence is close to $1$ and hence the algorithm {\bf PTA} will converge very slow. On the other hand, the last correction step in Algorithm \ref{alg31} is Newton's method. Therefore, it is quadratically convergent
in the step. This modification in (\ref{modif}) will not affect the performance of Algorithm \ref{alg31}.

The numerical results are reported in the following table. For Algorithm \ref{alg31}, \emph{Aiter} denotes the average number of
prediction-correction steps,  \emph{Anwtiter} denotes the average total number of Newton iterations and \emph{Atime} denotes the average elapsed CPU time in seconds. For the algorithm {\bf PTA}, \emph{Aiter} denotes the average number of iterations and \emph{Atime} denotes the average elapsed CPU time in seconds.

\begin{table}[h]
\centering
\setlength{\abovecaptionskip}{0pt}%
\setlength{\belowcaptionskip}{10pt}%
\caption{Performances of Algorithm \ref{alg31} and {\bf PTA} on Example \ref{eg5}}
\begin{tabular}{c c |c c c|c c}
\hline
& & \multicolumn{3}{c}{Algorithm \ref{alg31}} & \multicolumn{2}{|c}{\bf PTA}  \\
\hline
$(m,n)$ & $d$  & \emph{Aiter} & \emph{Anwtiter} & \emph{Atime} & \emph{Aiter} & \emph{Atime}\\
\hline
(3,10) & 3 & 11 & 14.05 & 0.0294 & 524.51 & 0.0272 \\
  (3,10) & 4 & 11 & 14.05 & 0.0298 & 5121.9 & 0.2382 \\
  (3,10) & 5 & 11 & 14.06 & 0.0299 & 49832 & 2.3989 \\
  (3,10) & 6 & 11 & 14.18 & 0.0303 & $>50000$ & \  \\
  \hline
  (4,10) & 4 & 11 & 14 & 0.0380 & 540.07 & 0.0426 \\
  (4,10) & 5 & 11 & 13.99 & 0.0373 & 5296.2 & 0.3353 \\
  (4,10) & 6 & 11 & 13.97 & 0.0371 & $>50000$ & \  \\
  \hline
  (3,20) & 4 & 11 & 14 & 0.0347 & 1321.6 & 0.0745 \\
  (3,20) & 5 & 11 & 14 & 0.0345 & 13146 & 0.6725 \\
  (3,20) & 6 & 11 & 14 & 0.0343 & $>50000$ & \  \\
  \hline
  (4,20) & 5 & 11 & 13.98 & 0.1945 & 706.8 & 0.1146 \\
  (4,20) & 6 & 11 & 13.66 & 0.1909 & $>50000$ &  \\
  (4,20) & 7 & 11 & 14.21 & 0.1955 & $>50000$ & \  \\
\hline
\end{tabular}
\label{tab2}
\end{table}

Table \ref{tab2} compares the average CPU time and the average number of iterations of these algorithms. From Table \ref{tab2}, we know that
the algorithm {\bf PTA} performs well when the parameter $d$ is
small. Algorithm \ref{alg31} is more efficient than the algorithm {\bf PTA}  when $d$ is large, in terms of number of iterations. The reason for
this is because the algorithm {\bf PTA}  is linearly convergent and its rate of convergence depends on the ratio in (\ref{rato}). When $d$ is large, the ratio is close to $0$ and the rate of convergence (\ref{rato}) is close to $1$. Thus the algorithm {\bf PTA} becomes slow. Numerical results show that the ratio in (\ref{rato}) does not affect the performance of Algorithm \ref{alg31}.
 Algorithm \ref{alg31} is better than the algorithm {\bf PTA} for large-scale problems.

\section{Conclusion}

 The dominant eigenpair of an essentially nonnegative tensor plays an important role in network resource allocations. In this paper, we first show that an irreducible essentially nonnegative tensor has a unique positive dominant eigenpair with the eigenvector normalized. On the other hand, according to Theorem \ref{thmdoeig} and \ref{uniqthm}, eigenpair $(\lambda(\mathcal{A})+\alpha,x(\mathcal{A}))$ coincides with Perron pair $(\rho(\mathcal{A}+\alpha\mathcal{I}),x)$ for nonnegative tensor $\mathcal{A}+\alpha\mathcal{I}$ with the eigenvectors normalized, which motivates us to find the dominant eigenpair via the homotopy method in \cite{Chen2019}. Based on the algorithmic framework of the homotopy method in \cite{Chen2019}, we propose a homotopy method for computing the dominant eigenpair. The convergence of the proposed homotopy method has been established for both irreducible and reducible cases in Theorems \ref{thmrank} and \ref{thmmain}. We have designed a concrete algorithm (Algorithm \ref{alg31}) using the Newton type prediction-correction strategy to implement the homotopy method. We have provided some numerical results to illustrate the efficiency of Algorithm \ref{alg31}. Numerical experiments show that the homotopy method is promising, particularly when the value of (\ref{rato}) is close to $1$, while the power-type method in \cite{Zhang2013} becomes slow for this case. Algorithm \ref{alg31} shows that the homotopy techniques are very useful for computing the dominant eigenpair of an essentially nonnegative tensor.



\end{document}